\documentclass{amsart}

\usepackage{amssymb,amsthm}
\usepackage{epsfig,color}

\newtheorem{theorem}{Theorem}[section]
\newtheorem{lemma}[theorem]{Lemma}
\newtheorem{prop}[theorem]{Proposition}
\newtheorem{cor}[theorem]{Corollary}

\theoremstyle{remark}
\newtheorem{remark}[theorem]{Remark}
\theoremstyle{definition}
\newtheorem{define}[theorem]{Definition}
\renewcommand\P{\mathbb{P}}
\newcommand\Z{\mathbb{Z}}
\newcommand\Q{\mathbb{Q}}
\newcommand\R{\mathbb{R}}
\newcommand\Qbar{\overline{\mathbb{Q}}}

\renewcommand\O{\mathcal{O}}

\newcommand\m{\mathfrak{m}}
\newcommand\n{\mathfrak{n}}
\newcommand\p{\mathfrak{p}}

\newcommand\Aut{\mathop{\rm Aut} \nolimits}
\newcommand\Pic{\mathop{\rm Pic} \nolimits}

\newcommand\main[1]{
\begin{theorem}\label{#1}
Let $a, b, c, d \in \Q^*$ be nonzero rational numbers with $abcd$ square.  
Let $P = (x_0:y_0:z_0:w_0)$ be
a rational point on $V_{a,b,c,d}$,
and suppose that $x_0 y_0 z_0 w_0 \ne 0$ and that $P$ does not
lie on any of the $48$ lines of the surface.  Then the set of 
rational points of the surface is dense in both the Zariski and the real 
analytic topology.
\end{theorem}
}

\newcommand\second[1]{
\begin{theorem}\label{#1}
Let $a, b, c, d \in \Q^*$ be nonzero rational numbers 
with $abcd$ square and $a+b+c+d=0$. 
Assume that no two of $a,b,c,d$ sum to $0$. Then the set of 
rational points of the surface $V_{a,b,c,d}$ is dense in both 
the Zariski and the real analytic topology.
\end{theorem}
}

\title{Density of rational points on diagonal quartic surfaces}
\author{Adam Logan, David McKinnon, Ronald van Luijk}

\renewcommand\footnote[1]{}

\begin{document}

\begin{abstract}
Let $a,b,c,d$ be nonzero rational numbers whose product is 
a square, and let $V$ be the diagonal quartic surface in $\P^3$ 
defined by $ax^4+by^4+cz^4+dw^4=0$. We prove that if $V$ 
contains a rational point that does not lie on any of the $48$
lines on $V$ or on any of the coordinate planes, then the set
of rational points on $V$ is dense in both the Zariski topology 
and the real analytic topology.
\end{abstract}

\maketitle

\section{Introduction}
\def\ct{Colliot-Th\'el\`ene}
This paper 
\footnote{please read introduction carefully again.}
is about the arithmetic of diagonal quartic surfaces, which are
the surfaces $V_{a,b,c,d} \subset \P^3$ defined by the equation
$ax^4 + by^4 + cz^4 + dw^4 = 0$ for nonzero $a,b,c,d\in \Q$. 
We will prove the following theorem.

\main{main}

We will also prove a generalization to arbitrary number fields of 
a weaker version of Theorem \ref{main}. 
An easy consequence of  Theorem \ref{main} is the following.

\second{second}

The surfaces $V_{a,b,c,d}$ are smooth quartic
surfaces, which means that they are K3 surfaces.  
One of the most important open problems in the arithmetic of K3
surfaces is to determine whether there is a K3 surface over a number
field on which the set of rational points is neither empty nor Zariski
dense.  Theorem \ref{main} shows that a diagonal quartic surface
over $\Q$ for which the product of the coefficients is a square does not
have this property, unless all its rational points lie on the union of 
its $48$ lines and the coordinate planes. 
However, no such diagonal surface is known 
and the authors believe that the condition in Theorem \ref{main} that 
$P$ not lie on one of 
the $48$ lines or on one of the coordinate planes may not be necessary.

Noam Elkies \cite{elkies} proved that the set of 
$\Q$-rational points on $V_{1,1,1,-1}$ is dense in both the 
Zariski topology and the real analytic topology.   
Martin Bright \cite{bright} has exhibited a Brauer-Manin obstruction
to the existence of rational points on many examples. 
Sir Peter Swinnerton-Dyer in his paper \cite {SDdiag} assumes like us 
that $abcd$ is a square. He uses one of the 
two elliptic fibrations that exist in this case to show that 
under certain specific conditions on the coefficients, 
$V_{a,b,c,d}$ does not satisfy the Hasse principle, while under some 
other hypotheses, including Schinzel's hypothesis and the assumption 
that Tate-Schafarevich groups of elliptic curves are finite, the
Hasse principle is satisfied. In particular, assuming these two big conjectures, 
it follows immediately from his work that if $abcd$ 
is a square but not a fourth power and no product of two coefficients 
or their negatives is a square and there is no Brauer-Manin obstruction 
to the Hasse principle, then the set of rational points is Zariski dense;
the last hypothesis is obviously satisfied when $V_{a,b,c,d}(\Q)$ is 
nonempty. By Theorem \ref{main} the fact that the set $V_{a,b,c,d}(\Q)$ is 
nonempty indeed implies that it is Zariski dense, independent from 
Schinzel's hypothesis and the assumption that Tate-Schafarevich 
groups of elliptic curves are finite, provided that
we assume instead the existence of a rational point that does not 
lie on any of the $48$ lines or any of the coordinate planes. 

Jean-Louis \ct{}, Alexei Skorobogatov, and Sir Peter Swinnerton-Dyer 
\cite{ctskorswin} also use Schinzel's hypothesis and finiteness 
of Tate-Schafarevich groups
to show that over arbitrary number fields, on semistable elliptic 
fibrations 
satisfying certain technical conditions, the Brauer-Manin obstruction
coming from the vertical Brauer group is the only obstruction to the 
Hasse principle; furthermore, that
if such a fibration contains a rational point, then its set 
of rational points is Zariski dense. 
Olivier Wittenberg \cite{wit} generalizes their theory to the extent that 
Sir Peter Swin\-nerton-Dyer's aforementioned result over the rational numbers 
becomes a special case of this more general setting, thus extending the
result to arbitrary number fields.

Jean-Louis \ct{}\footnote{ask for reference} pointed out Richmond's 
method \cite{richmond} 
to the authors that takes a rational point $P$ on $V=V_{a,b,c,d}$
to construct two new points over $\Q(\sqrt{abcd})$.
Each of these two points is the unique last point of intersection 
between $V$ and one of the two tangent lines to the singular node in 
the intersection between $V$ and the tangent plane to $V$ at $P$.
In this paper we reinterpret this construction to study the 
arithmetic of the surface $V$.

In the next section,
we exhibit two endomorphisms $e_1$ and $e_2$ of $V_{a,b,c,d}$
such that $e_1(P)$ and $e_2(P)$ are the two points given by Richmond's 
construction. The diagonal surfaces have two elliptic fibrations and 
each fibration is fixed by one of the two
endomorphisms.  Thus, if $e_i$ is one of the endomorphisms and $P$ is a rational
point on the surface, we will consider the fibre $C_i$ of the fibration
fixed by the endomorphism $e_i$ that passes through $P$.  This is a curve,
so we can study the divisor $(e_i(P)) - (P)$ on it.
Subject to the hypotheses, we will see
that it is almost never a torsion divisor, and hence that fibres with rational
points tend to have positive rank.  

Our results are very much in the spirit of the potential density
results of Bogomolov and Tschinkel \cite{elliptic, ellipticK3}, and Harris and
Tschinkel \cite{quartic}.  These papers describe a variety of techniques
for proving density and potential density of rational points on a
variety of surfaces, including in particular the diagonal quartic
surfaces we consider in this paper.  Our results improve on these only
in that we strengthen the conclusion of the potential density results
to actual density, and that we weaken the hypotheses on the density
results to demanding only a single rational point satisfying a weak
genericity condition.  For an excellent overview of techniques used to
prove density and potential density of rational points on algebraic
varieties, please see Brendan Hassett's survey \cite{hassett}.

We thank Jean-Louis \ct, Samir Siksek, Andrew Granville, and Igor Dolgachev for useful discussions.

\section{The elliptic fibrations and
endomorphisms}

We begin by introducing some notation.

\begin{define}
For $a,b,c,d \in \Q^*$ we let $V_{a,b,c,d}$ be the surface in 
$\P^3$ given by $ax^4+by^4+cz^4+dw^4=0$. Set $V_0 = V_{1,1,1,1}$ and 
$V_0' = V_{1,1,-1,-1}$.  Let $\tau \colon \P^3 
\rightarrow \P^3$ be the map that squares all four coordinates.
Set $Q_{a,b,c,d} = \tau(V_{a,b,c,d})$.
\end{define}

Suppose $a,b,c,d \in \Q^*$ with $abcd\in (\Q^*)^2$ and write 
$V$ and $Q$ for $V_{a,b,c,d}$ and $Q_{a,b,c,d}$ respectively.
Suppose that $V$ has a rational point.
Then $Q$, which is a nonsingular quadric surface defined by
$ax^2+by^2+cz^2+dw^2=0$, also contains 
a rational point.  Since $abcd$ is a square, the
two rulings on $Q$ are defined over $\Q$.

\begin{define} Fix a rational point $R$ on $Q$, and decompose the 
intersection of $Q$ with the tangent plane to $Q$ at $R$ into two lines $l_1$, $l_2$.
Let $\pi_1, \pi_2 \colon Q \rightarrow \P^1$
be two rulings on $Q$ such that $l_i$ is a fibre of $\pi_i$.
For $i = 1, 2$, set $f_i = \pi_i \circ \tau \colon V \rightarrow \P^1$.
\end{define}

Our description of the two rulings shows that they can be defined over $\Q$. 
However, the two rulings do not depend on $R$ in the following sense.
Let $R'$ be another rational point on $Q$. Then by the same construction we obtain 
two rulings, which we can number $\pi_1', \pi_2' \colon Q \rightarrow \P^1$, 
such that for each $i$ the maps $\pi_i$ and $\pi_i'$ coincide up to a 
linear automorphism of $\P^1$. 

Any two linear forms defining $l_i$ define a map to $\P^1$ equal to
$\pi_j$ up to a linear transformation of $\P^1$, with $i \neq j$.
Using additional lines in the same family, we can obtain alternative
equations for the rulings and remove the base locus. The $f_i$ are
elliptic fibrations on $V$, also well-defined up to an automorphism of $\P^1$. 

\begin{define}\label{iota}
Fix fourth roots of $a,b,c,d$.  Let $\iota_{a,b,c,d}$ be the
$\Qbar$-isomorphism $V_{a,b,c,d} \rightarrow V_0$ defined by $(\root 4
\of a x: \root 4 \of b y: \root 4 \of c z: \root 4 \of d w)$.  Fix
fourth roots of $1$ and $-1$, and let $\iota'_{a,b,c,d}$ be the
$\Qbar$-isomorphism $V_{a,b,c,d} \rightarrow V_0'$ defined by
${\iota_{1,1,-1,-1}}^{-1} \circ \iota_{a,b,c,d}$.
\end{define}

Note that the fibrations $f_i$ of $V=V_{a,b,c,d}$ were constructed 
geometrically. Therefore, the fibrations of $V_{a,b,c,d}$ coincide 
with those of $V_{1,1,1,1}$ up to composition with $\iota_{a,b,c,d}$
and a linear automorphism of $\P^1$.
 
When we study the geometric properties of diagonal quartics, it
suffices to consider $V_0$ or $V_0'$.  It is only when we consider the
arithmetic properties that we need to allow the coefficients to vary.
While some formulas are more symmetrical on $V_0$, some things are
defined over $\Q$ for $V_0'$ that are not for $V_0$.  For example, the
two elliptic fibrations on $V_0'$ are defined by $(x^2-z^2:y^2-w^2)$
and $(x^2-w^2:y^2-z^2)$, whereas on $V_0$ they can only be described
over $\Q$ as maps to a curve isomorphic to the conic $x^2+y^2+z^2 =
0$.  To give a fibration of $V_0$ over $\P^1$ requires changing base
to a field over which this conic has a point.

\begin{define}
Let $\mu$ denote the group of automorphisms of $\P^3$ that multiply
each coordinate by a fourth root of unity, and let $S_4$ act on the
coordinates of $\P^3$.  We will regard $\mu$ as inducing a subgroup of
$\Aut V$.  Any permutation $\pi \in S_4$ induces an isomorphism from
$V_{a,b,c,d}$ to $V_{a',b',c',d'}$, where $(a',b',c',d')$ is the
appropriate permutation of $(a,b,c,d)$.
\end{define}

\begin{define}
Let $G$ be the semidirect product
$\mu \rtimes S_4$, with the obvious action of $S_4$ on $\mu$.
We will view $G$
as a subgroup of $\Aut V_0$, and through conjugation with $\iota_{a,b,c,d}$
also as a subgroup of $\Aut V_{a,b,c,d}$.
\end{define}

Note that when $G$ is viewed as acting on $V_0$, the elements of $S_4$
correspond to $\Q$-automorphisms of $V_0$; this is not the case when
$G$ is considered as acting on a general $V_{a,b,c,d}$.

The surface $V$ contains exactly $48$ lines, on which $G$ acts transitively. 
On $V_0'$ one of these lines is given by $x=z$ and $y=w$. 

\begin{define}
Let $G_0$ denote the index-$2$ subgroup of $G$ that fixes the fibrations $f_i$ (up to an 
automorphism of $\P^1$).
\end{define}

The group $G_0$ partitions the $48$ lines into two orbits 
$\Lambda_i$ of size $24$ (with $i=1,2$), where $\Lambda_i$ consists of the 
irreducible components of the $6$ singular fibres of $f_i$, each being of type $I_4$
\cite[page 517]{SDdiag}. The singular points of the fibres are 
exactly the $24$ points with two coordinates zero, and each 
of these points is singular on its fibre in both fibrations. 

\begin{define}Let $\Omega$ denote the set of these $24$ points, and let 
$U$ be the complement of $\Omega$ in $V$.
\end{define}

We will see that the ``tangents'' to the node at $P$ that we described in 
the introduction can be characterized in a different manner as well, 
namely as the tangents to the fibres of the $f_i$ through $P$.
We first show that these tangents do not interfere too much with 
the singular fibres.

\begin{lemma}\label{nointerference}
Fix a point $P \in U(\Qbar)$, and $i,j$ such that $\{i,j\} = \{1,2\}$. 
For $k=1,2$, let $C_k$ be the fibre of $f_k$ through $P$, and let $M$ 
be the tangent line to $C_j$ at $P$. Then $M$ is not contained in $C_i$.
\end{lemma}
\begin{proof}
Since this statement is completely geometric, we assume $V=V_0'$ without 
loss of generality. Note that $M$ is well-defined because $C_j$ is smooth 
at $P$. 
Suppose $M$ is contained in $C_i$. Then $M$ is one of the $48$ lines.
After acting on $V$ by an appropriate element of $G$, the line $M$ is given by 
$x=z$ and $y=w$, so there are $s,t \in \Qbar$ such that $P = (s:t:s:t)$.
Since $M$ is contained in the fibre above $(0:1)$ of the fibration that 
sends $(x:y:z:w)$ to $(x^2-z^2:w^2+y^2)$, the curve $C_j$ is a fibre of 
the other fibration, so $f_j$ can be given by $(x^2+z^2:w^2+y^2)$, or 
equivalently $(w^2-y^2:x^2-z^2)$. 
Since $f_j(P) = (s^2:t^2)$, we conclude that $C_j$ is given by 
$s^2(w^2+y^2) = t^2(x^2+z^2)$ and $t^2(w^2-y^2)=s^2(x^2-z^2)$. 
The tangent line to $C_j$ at $P$ is therefore also given by 
$s^2t(w+y) = t^2s(x+z)$ and $t^3(w-y) = s^3(x-z)$. Simple linear 
algebra shows that this does not contain $M$ unless $st=0$.
This contradicts the assumption that
$P\in U$, which shows that $M$ is not contained in $C_i$. 
\end{proof}

The following proposition is fundamental to our work and 
shows how the case of diagonal $V$ is special.

\begin{prop}\label{essence}
Fix a point $P \in U(\Qbar)$ and set $R = \tau(P)$. Let 
$T_R$ denote the tangent space to $Q$ at $R$, and set 
$A = \tau^{-1}(T_R)$. Fix $i \in \{1,2\}$, and 
let $C_i$ be the fibre of $f_i$ through $P$.
Let $M$ denote the 
tangent line to $C_i$ at $P$. Then $M$ is contained in $A$. 
Furthermore, let $T_P$ denote the tangent plane to $V$ at $P$.
If $C_i$ is irreducible, then the intersection multiplicities 
$(M \cdot (T_P \cap V))_P$  and $(T_P \cdot C_i)_P$ are at least $3$. 
%
%
\end{prop}
\begin{proof}
%
%
Note that $C_i = \tau^{-1}(L_i)$, where $L_1$ and $L_2$ are the lines 
in $T_R \cap Q$, so we have 
$C_1 \cup C_2 = \tau^{-1}(T_R \cap Q) = \tau^{-1}(T_R) \cap \tau^{-1}(Q) = A \cap V$. 
By the assumption $P \in U$, the curve $C_i$ is smooth at $P$, so 
$M$ is well-defined. Without loss of generality, we assume that $P$ 
is contained in the affine part $w=1$, given by $P=(x_0,y_0,z_0)$. 
Since the statement of the lemma is completely geometric, we 
may assume that $Q$ is given by $x^2+y^2+z^2+1=0$, so that $V$ is 
given by $q_V=0$ with $q_V = x^4+y^4+z^4+1$, and $A$ by $q_A=0$ 
with $q_A = x_0^2x^2+y_0^2y^2+z_0^2z^2+1$. Note that at most one of
the coefficients of the equation defining $A$ is $0$, so $A$ is 
irreducible and smooth at $P$. The common tangent space $T_P$
to $V$ and $A$ at $P$ is given by $l=0$, where
$$l=x_0^3(x-x_0)+y_0^3(y-y_0)+z_0^3(z-z_0) = 
x_0^3x+y_0^3y+z_0^3z+1.$$ 

It turns out that since $Q$ is diagonal, 
the surfaces $A$ and $V$ are more similar locally at $P$ than is
implied by the fact that they share 
a tangent space. Let $\O_{\P^3,P}$ and $\m$ be the local ring 
of $P$ in $\P^3$ and its maximal ideal. 
Set $g=x_0^2(x-x_0)^2+y_0^2(y-y_0)^2+z_0^2(z-z_0)^2 \in \m^2$. 
Then the quadratic approximations of $q_A$ and $q_V$ are 
$q_A\equiv 2l+g {\mod {\m^3}}$ and 
$q_V\equiv 4l+6g {\mod {\m^3}}$. Let $q_1,q_2\in k[x,y,z]$ be 
quadrics such that 
$C_i$ is given on $A$ by $q_i=0$. 
From $C_1 \cup C_2 = V \cap A$ we conclude that 
$q_V \equiv cq_1q_2 {\mod q_A}$ for some nonzero constant $c$.
Replacing $q_1$ by $cq_1$, we find that
there exists a quadric $r\in k[x,y,z]$ such that $q_V = q_1q_2 +q_Ar$. 
From $q_i \in \m$ we find $4l \equiv q_V \equiv q_Ar \equiv 2lr {\mod {\m^2}}$,
and since $2l \ne 0$ in $\m/\m^2$, this implies that
$r \equiv 2 {\mod {\m}}$. 

Let $\O_{M,P}$ and $\n$
denote the local ring of $P$ on $M$ and its maximal ideal, and let 
the reduction map $\O_{\P^3,P}\rightarrow \O_{M,P}$ be given by 
$s \mapsto \overline{s}$. Since $M$ is contained in $T_P$, we have 
$\overline{l}=0$. Since $C_i$ is tangent to $M$, we have 
$\overline{q_i} \in \n^2$, so $\overline{q_1}\overline{q_2} \in \n^3$. 
Note that this also holds in case $M$ is a component of $C_i$, 
because then we have $\overline{q_i}=0$. Therefore, we find 
$6\overline{g} \equiv \overline{q_V} \equiv \overline{r}\overline{q_A} 
\equiv 2\overline{g}  {\mod {\n^3}}$, and so
$\overline{g} \in \n^3$. This implies $\overline{q_A} =\overline{g} 
\in \n^3$. If $M$ were not contained in 
$A$, then this would imply that $A$ intersects $M$ at $P$ with 
multiplicity at 
least $3$, which is impossible because $A$ is quadratic and $M$ is 
linear. We conclude that $M$ is indeed contained in $A$.

Now assume that $C_i$ is irreducible. Then by Lemma \ref{nointerference} 
the line $M$ is not contained in $C_1 \cup C_2 = A \cap V$, so 
$M \not \subset V$. This implies that
the first intersection is between 
curves that have no common components.  Because $C_i$ is not 
contained in a hyperplane, the same holds for the second intersection.

The first intersection takes place in $T_P$ and the multiplicity
equals the valuation of $\overline{q_V}$ in $\O_{M,P}$.  From the
congruence $\overline{q_V} \equiv 6\overline{g} {\mod {\n^3}}$ we
conclude that $\overline{q_V} \in \n^3$, so the multiplicity is at
least $3$. Let $\O_{C,P}$ and $\p$ denote the local ring of $P$ on
$C_i$ and its maximal ideal, and let the 
reduction map $\O_{\P^3,P}\rightarrow \O_{C,P}$ be given by $s \mapsto
\tilde{s}$. Since $\tilde{q_V}$ vanishes on $C_i$, we have
$4\tilde{l}+6\tilde{g} \equiv 0 {\mod{\p^3}}$.  Because $\tilde{q_A}$
vanishes on $C_i$, we also have $2\tilde{l}+\tilde{g}=0$. Together
this implies $l \in \p^3$. As $l$ defines $T_P$, this implies that
$T_P$ intersects $C_i$ at $P$ with multiplicity at least $3$.
\end{proof}

\begin{define}
For the rest of this section, fix a point $P \in U(\Qbar)$ and
$i,j$ such that $\{i,j\} = \{1,2\}$.  For $k=1,2$, let $C_k$ be the fibre of
$f_k$ through $P$ and $M_k$ the tangent line to $C_k$ at $P$.
\end{define}

\begin{cor}\label{otherpoint}
The line $M_j$ intersects $C_i$ in a scheme of dimension $0$ and
degree $2$ that contains the reduced point $P$.  (The intersection
may be a nonreduced scheme supported at $P$.)
\end{cor}
\begin{proof}
Let $T_R$ and $A=\tau^{-1}(T_R)$ be as in Proposition \ref{essence}. Recall
that $C_i=\tau^{-1}(L)$ for some line $L$ in $T_R$, so $C_i$ is 
defined in $A$ by a single quadric. The line $M_j$ is
contained in $A$ by Proposition \ref{essence}, but not in $C_i$
by Lemma \ref{nointerference}, so the dimension of the intersection is 
indeed $0$. First suppose $A$ is nonsingular.
Then $M_j$ is a line in one of the rulings on $A$. The curve $C_i$ on $A$ is 
of type $(2,2)$, so it intersects $M_j$ twice, counted 
with multiplicity. The claim follows immediately. Now assume that 
$A$ is singular. Without loss of generality we assume $V=V_0$. 
Suppose $P = (x_0:y_0:z_0:w_0)$, so that $A$ is given by 
$x_0^2x^2+y_0^2y^2+z_0^2z^2+w_0^2w^2=0$. The facts that $A$ is singular 
and that $P \in U$ imply that exactly one of the coordinates 
of $P$ is zero. We deduce that $A$ is the cone over a smooth conic, 
and that $P$ is not contained in any of the lines on $V$, as their 
equations imply that if one coordinate is zero, then so is another. 
Therefore $C_i$ is smooth, so, as $C_i$ is defined in $A$ by a single 
equation, it does not contain the vertex $S$ of $A$. This means we can 
naively apply the usual intersection theory to study the intersections
of $C_i$ with other curves.  The line $M_j$ is a line on $A$ through the 
vertex $S$ and intersects any hyperplane section that does not 
go through $S$ once. Since $C_i$ is a quadratic hypersurface section, 
we find $M_j \cdot C_i = 2$. Again, the claim follows.
\end{proof}

\begin{define}
Following Corollary \ref{otherpoint}, we define two morphisms
$e_i,e_j \colon U \rightarrow V$. The morphism $e_i$ sends
any point $R$ to the unique second intersection point between the
fibre of $f_i$ through $R$ and the tangent at $R$ to the fibre of
$f_j$ through $R$. The morphism $e_j$ is defined by interchanging 
the roles of $i$ and $j$.
\end{define}

By definition, $e_i$ and $e_j$ respect the fibrations $f_i$ and $f_j$ 
respectively. 

\begin{cor}\label{hyp}
Let $T_P$ 
denote the tangent plane to $V$ at $P$, and set 
$D = T_P \cap V$. If $C_j$ is irreducible, then 
the line $M_j$ intersects $D$ in the 
divisor $3(P)+(e_i(P))$ of $D$. If $C_i$ is irreducible, then $T_P$ intersects 
$C_i$ also in $3(P)+(e_i(P))$, but as a divisor on $C_i$.
\end{cor}
\begin{proof}
Let $A$ be as in Proposition \ref{essence}. Then we have 
$M_j \subset A$ and $A \cap V = C_i \cup C_j$. If $C_j$ is 
irreducible, then $M_j$ is not contained in $C_j$, and 
by Lemma \ref{nointerference} also not in $C_i$, so $M_j \not 
\subset V$, and $M_j \cap D$ is $0$-dimensional. If $C_i$ is 
irreducible, then it is not contained in $T_P$, so $T_P \cap C_i$ 
is $0$-dimensional.
Both intersections have degree $4$ by B\'ezout's Theorem, applied 
in $T_P$ and $\P^3$ respectively. By Proposition \ref{essence} the 
intersection at $P$ has multiplicity at least $3$, so it suffices to 
show that $e_i(P)$ is contained in both intersections. This follows 
from $e_i(P) \in C_i \subset V$ and $e_i(P) \in M_j \subset T_P$. 
\end{proof}

\begin{remark}
Since the $M_k$ intersect $D$ at $P$ with multiplicity at least $3$, they 
are exactly the ``tangent'' lines to the node on $D$ at $P$ that we
discussed in the introduction.
By Corollary \ref{hyp} this means that the $e_i(P)$ are the two points 
obtained from $P$ as described there. 
\end{remark}

\begin{remark}
Let $H$ be a hyperplane section of the generic fibre 
$\mathcal{V}_i/{\Qbar(t)}$ 
in $\P^3$ of $f_i$. By Proposition \ref{hyp}, if we identify 
$\mathcal{V}_i$ with $\Pic^1(\mathcal{V}_i)$, then $e_i$ is given 
by sending $R$ to $H-3R$ for any point $R$ on $\mathcal{V}_i$. 
\end{remark}


Although a computer is useful, even by hand it is not impossible
to check that, on an open subset, $e_1$ and $e_2$ are given by sending 
$(x_0:y_0:z_0:w_0)$ to $(x_1:y_1:z_1:w_1)$ with
\begin{align}
x_1&=x_0\big((3bcy_0^4z_0^4+adx_0^4w_0^4)(ax_0^4+dw_0^4) + 4Nx_0^2y_0^2z_0^2w_0^2(by_0^4-cz_0^4)\big),\nonumber\\
y_1&=y_0\big((3acx_0^4z_0^4+bdy_0^4w_0^4)(by_0^4+dw_0^4) + 4Nx_0^2y_0^2z_0^2w_0^2(cz_0^4-ax_0^4)\big),\label{abcdeqs}\\
z_1&=z_0\big((3abx_0^4y_0^4+cdz_0^4w_0^4)(cz_0^4+dw_0^4) + 4Nx_0^2y_0^2z_0^2w_0^2(ax_0^4-by_0^4)\big),\nonumber\\
w_1&=w_0\big(cdz_0^4w_0^4(cz_0^4+dw_0^4) - abx_0^4y_0^4(9cz_0^4+dw_0^4)\big),\nonumber
\end{align}
where $N$ is one of the two square roots of $abcd$.

\begin{define}
For each coordinate $v$ of $\P^3$, let $\sigma_v \in \mu$ denote the 
automorphism of $\P^3$ that negates the $v$-coordinate.  For a pair
of coordinates $u,v$, let $\sigma_{uv} = \sigma_u \sigma_v$.  
\end{define}

\begin{prop}\label{sigmacomm}
The automorphisms $\sigma_u$ commute with the maps $e_i$ on $V$ and with 
all the maps $\iota_{a,b,c,d}$.
\end{prop}
\begin{proof}
For the $e_i$, this follows immediately from the equations in (\ref{abcdeqs}).
For the $\iota_{a,b,c,d}$ it is obvious.
\end{proof}

\begin{prop}\label{onlines}
Assume $P$ is contained in a line that is an 
irreducible component of $C_i$. Then 
there are two different coordinates $u,v$ of $\P^3$ such that 
$e_j(P) = \sigma_{uv}(P)$, while $P$ and $e_i(P)$ lie on 
nonintersecting components of $C_i$ and $e_i(P) \not \in \Omega$.   In
particular, $e_i(P)$ lies on a line.
\end{prop}
\begin{proof}
By Proposition \ref{sigmacomm} we may assume $V=V_0'$. 
Let $L$ be the line in $C_i$ that contains $P$. Note that the 
group $G$ acts by conjugation on the set of the three automorphisms 
of the form $\sigma_{uv}$ with $u \neq v$. 
The group $G$ also acts by conjugation on the pair 
of fibrations and acts accordingly on the set of the two sets 
$\Lambda_i$ of lines. It follows that after acting with an appropriate 
element of $G$, we may assume that $L$ is given by $x=z$ and $y=w$, 
so there are $s,t \in \Qbar$ such that $P=(s:t:s:t)$. From $P \not \in
\Omega$ we get $st \neq 0$. Then $f_i$ can 
be given by $(x:y:z:w) \mapsto (x^2-z^2:w^2+y^2)$, while $f_j$ 
can be given by $(x:y:z:w) \mapsto (x^2+z^2:w^2+y^2)$. 
As $e_i$ respects $f_i$, it is easy to check which equations in (\ref{abcdeqs})
give $e_i$. 
It turns out that $e_i$ is given by (\ref{abcdeqs}) with $N=1$, 
while $e_j$ is given by $N=-1$. Substituting in (\ref{abcdeqs}),
we find $e_j(P) = (-s:t:-s:t) = \sigma_{xz}(P)$ and 
$e_i(P) = (t^3:-s^3:-t^3:s^3)$. It is clear that $e_i(P)$ is not
contained in $\Omega$ and lies on the component of $C_i$ given by 
$x+z=y+w=0$, which is a line and does not intersect $L$. 
\end{proof}


\begin{prop}\label{permute}
Let $\pi \in S_4$ be an automorphism of $V_0$ in $\P^3$ given by permutation of the
coordinates, and let $S_4$ act on $V=V_{a,b,c,d}$ by conjugating 
the action on $V_0$ with $\iota_{a,b,c,d}$. 
Then $\pi e_i = e_k \pi$, where $k = i$ if the permutation
underlying $\pi$ is even and $k = j$ if it is odd.
\end{prop}
\begin{proof}
The statement is purely geometric, so we may assume $V=V_0$.
Note that $\pi$ switches the rulings on the quadric given by 
$x^2+y^2+z^2+w^2=0$ if and only if $\pi$ is odd, and it permutes 
the two elliptic fibrations on $V$ accordingly.
The statement therefore follows from Corollary \ref{otherpoint}.
\end{proof}

We can rephrase Proposition \ref{permute} by stating that inside the
group $G = \mu \rtimes S_4$ we have $G_0 \cap S_4 = A_4$.  We see that
if we conjugate the equations for $e_i$ in (\ref{abcdeqs}) by an element in
$A_4$, then we obtain a new set of equations for $e_i$.  In
particular this will be true for the subgroup $V_4 \subset A_4$ of
products of disjoint cycles. If instead we conjugate the
equations for $e_i$ by an odd element of $S_4$, we get a new set of equations
for $e_j$.

\begin{prop}\label{extend}
The map $e_i$ extends to an endomorphism of $U$.
The sets of equations obtained by conjugating those in (\ref{abcdeqs})
with elements of $V_4$ are sufficient to define $e_i$ on all of $U$.
\end{prop}
\begin{proof}
The set of all points $R$ where none of these $4$ 
sets of equations defines a regular map is determined on $V$ 
by the vanishing of all $16$ polynomials in the sets.
One checks by computer that this set is supported on $\Omega$.
Now suppose that for some $R$ we have $e_i(R) \in \Omega$, and let $C_i$ 
be the fibre of $f_i$ through $R$. Then $e_i(R)$ also lies on $C_i$, so 
$C_i$ is a singular fibre, and Proposition \ref{onlines} contradicts 
$e_i(R) \in \Omega$. We conclude that $e_i(U) \subset U$. 
\end{proof}

\begin{remark}\label{noextend}
In fact the $e_i$ do not extend to any of the points in $\Omega$. 
One way to prove this is to show that, for each point $R\in \Omega$ and
each $i$, there are two lines that meet at $R$ whose images under the
$e_i$ are disjoint. This reflects the fact that K3 surfaces do not
have endomorphisms of degree greater than $1$. 
\end{remark}

\begin{prop}\label{zerofix}
The following are equivalent: (a) $e_1(P) = P$,
(b) $e_2(P) = P$, (c) exactly one of the coordinates of $P$ is $0$.
\end{prop}
\begin{proof} 
Without loss of generality we suppose $V=V_0$. Let $A$ be as in
Proposition \ref{essence} and assume $P=(x_0:y_0:z_0:w_0)$. Then $A$
is given by $x_0^2x^2+y_0^2y^2+z_0^2z^2+w_0^2w^2=0$. Suppose (c)
holds.  Then $A$ is the cone over a nonsingular conic, and from
$M_1,M_2 \subset A$ we conclude that $M_1=M_2$ is the unique line on
$A$ through $P$ and the vertex of $A$. This implies that the $M_i$ are
both tangent to both the $C_i$, so the second intersection point in
$M_1 \cap C_2$ and $M_2 \cap C_1$ is again $P$, proving (a) and
(b). Alternatively, we can verify the statement by direct calculation:
we may assume $x_0=0$ by Proposition \ref{permute}. Substituting into
(\ref{abcdeqs}) and using $-y_0^4 =z_0^4 + w_0^4$, we simplify
$e_i(P)$ to
%
%
$(0:-y_0^5 z_0^4 w_0^4:-y_0^4 z_0^5 w_0^4:-y_0^4 z_0^4 w_0^5) =
(0:y_0:z_0:w_0)$.  Conversely, fix $k \in \{1,2\}$ and assume 
$e_k(P) = P$.  
If $C_k$ were singular, then $P$ and $e_k(P)$ would lie on nonintersecting
components of $C_k$ by Proposition \ref{onlines}, so we conclude that $C_k$ 
is nonsingular. By 
Corollary \ref{hyp} the divisor $4(P)=3(P) +(e_k(P))$ on $C_k$ is linearly 
equivalent to a hyperplane section $H$. Since multiplication by $4$ on the 
Jacobian of $C_k$ has degree $16$, there are exactly $16$ points $S$ on 
$C_k$ for which $4S$ is linearly equivalent to $H$. By the implication 
(c) $\Rightarrow$ (a) \& (b), we can already account for all $16$ of these 
points, namely the intersection points of the coordinate planes with $C_k$, 
as each of the $4$ planes intersects $C_k$ in $4$ different points. 
This shows that $P$ is one of these points, which proves (c). 
\end{proof}


\begin{prop} \label{esquared}
Suppose that $C_i$ is irreducible. Then the divisor $(e_i^2(P))-(P)$ 
on $C_i$ is linearly equivalent to $-2(e_i(P))+2(P)$.  
\end{prop}
\begin{proof}
Let $H$ denote the hyperplane class on $C_i$.  By Corollary  
\ref{hyp} we have $3(S)+(e_i(S)) \sim H$ for any point $S$ on $C_i$.  
Applying this to $S=P$ and $S=e_i(P)$, we obtain $3(P)+(e_i(P)) \sim 
3(e_i(P)) +(e_i^2(P))$, from which the statement follows.
\end{proof}

\begin{remark}
Proposition \ref{esquared} tells us that if we give $C_i$ the structure 
of an elliptic curve with origin $P$, then we have $e_i^2(P) = -2e_i(P)$, 
and by induction we obtain $e_i^n(P) = a_ne_i(P)$ with $a_1=1$ and 
$a_{n+1} = -3a_n+1$. 
\end{remark}

\begin{cor}\label{orderthree}
Assume that $C_i$ is irreducible. Then the divisor 
$(e_i(P)) - (P)$ has order dividing $3$ if and only if $e_i^2(P) = e_i(P)$.  
\end{cor}
\begin{proof}
By Proposition \ref{esquared}, the order divides $3$ if and only 
if $e_i^2(P)$ is linearly equivalent to $e_i(P)$.  This is equivalent to 
$e_i^2(P)=e_i(P)$, because two different points cannot be linearly equivalent 
on a curve of positive genus.  
\end{proof}

Our goal is to prove that the class of the divisor $(e_i(P)) - (P)$ on $C_i$ 
is often of infinite order in the Jacobian of $C_i$. It turns out that this class is
$2$ times the class of a divisor that has a simple description.

\begin{prop}\label{phipistwo}
Assume that $C_i$ is nonsingular.  Then
the divisors $(e_i(P))-(P)$ and $2(\sigma_u P) -2(P)$ on $C_i$ are linearly
equivalent, where $u$ is any of the coordinates on $\P^3$.  
\end{prop}

\begin{proof}
As this statement is purely geometric, we may assume $V=V_0'$. 
By Proposition \ref{permute}, we may assume that $f_i$ is given by 
$(x:y:z:w) \mapsto (x^2-z^2:y^2-w^2)$, and since $A_4$ acts transitively 
on the coordinates, also that $u=y$. 
%
%
Adding $4(P)$ to both sides
and using Corollary \ref{hyp} to identify $(e_i(P)) + 3(P)$ with
the hyperplane class on $C_i$,
we see that it is enough to find a hyperplane
whose intersection with $C_i$ is $2(P) + 2(\sigma_y(P))$.  
The plane $(s^2+t^2)(w_0x-x_0w) - (s^2-t^2)(w_0z-z_0w)$ has
this property, where $(s:t) = (x_0^2-z_0^2:y_0^2-w_0^2)$ is the
image of $P$ under $f_i$.
\end{proof}

\begin{prop} \label{twotorsion}
Assume $C_i$ to be irreducible. Then for any two coordinates $u,v$ 
of $\P^3$, the divisor $2(\sigma_{uv}(P))-2(P)$ on $C_i$ is principal.  
\end{prop}
\begin{proof}
Twice applying Proposition \ref{phipistwo} to $\sigma_u(P)$, once with
the coordinate $u$ and once with $v$, we see that the divisors
$2(\sigma_{uv} (P)) - 2 (\sigma_u (P))$ and $2(P) - 2 (\sigma_u (P))$
are linearly equivalent.  The result follows immediately.
\end{proof}

\begin{prop}\label{fourtorsion}
Let $P=(x_0:y_0:z_0:w_0)$ and $V=V_0'$.  Let $C_1$ be the fibre
through $P$ of the fibration given by $(x:y:z:w) \mapsto (x^2+z^2:
y^2+w^2)$, and let $C_2$ be the other fibre.  For $k=1,2$, we let
$E_k$ be the elliptic curve $C_k$ with $P$ as origin, provided that
$C_k$ is nonsingular.  Let $i$ denote a fixed square root of $-1$.  If
$C_1$ is nonsingular, then the rational $4$-torsion points on $E_1$
are the $(\pm z_0: \pm w_0: \pm x_0: \pm y_0)$ with three $+$ signs
and one $-$ sign; the other $4$-torsion points are given by
\begin{eqnarray*}
(w_0: \pm iz_0: \pm iy_0: x_0),& (-w_0: \mp iz_0: \pm iy_0: x_0),\\
(y_0: \pm ix_0: \pm iw_0: -z_0),& (y_0: \pm ix_0: \mp iw_0: z_0).\\
\end{eqnarray*}
If $C_2$ is nonsingular, then on $E_2$, the $(\pm w_0, \pm z_0, \pm y_0, \pm x_0)$ with three
$+$ signs and one $-$ sign are all the rational $4$-torsion points; the
others are given by
\begin{eqnarray*}
(\pm iz_0: w_0: \pm ix_0: y_0),& (\pm iz_0: -w_0: \mp ix_0: y_0),\\
(\mp iy_0: x_0: \pm iw_0: z_0),& (\pm iy_0: -x_0: \pm iw_0: z_0).\\
\end{eqnarray*}
In particular, for each $k\in \{1,2\}$ there is a set $S_k$ of $12$ 
automorphisms of $V_0'$ defined over $\Q(i)$ such that, for $R$ in an
open subset of $V_0'$,
the class of the divisor $(e_k(R)) - (R)$ on the fibre of $f_k$ through 
$R$ is of exact order $4$ if and only if
$e_k(R) = \alpha(R)$ for some $\alpha \in S_k$.
\end{prop}
\begin{proof}
One proof consists of finding a Weierstrass form for $C_i$, finding 
all $4$-torsion points, and then pulling them back to our model.
Instead, we show directly that for each $S$ among the given points, 
the double $2S$ is one of the $2$-torsion points $\sigma_{uv}(P)$ 
given in Proposition \ref{twotorsion}.  Take for instance the point
$S=( z_0: w_0: x_0: - y_0)$ on $E_1$.  Applying $\sigma_y$ and then 
the appropriate permutation to the function $m/l$ in the proof of 
Proposition \ref{twotorsion}, we find that the function 
$$
g = \frac{z_0w_0(w_0x-z_0y)-x_0y_0(y_0z+x_0w)}
         {z_0^3x+w_0^3y-x_0^3z+y_0^3w}
$$
is regular everywhere, except for a double pole at $S$.  Set $\lambda = g(P)$, 
then $g-\lambda$ also vanishes at $\sigma_{xz}(P)$, so we have 
$(g-\lambda) = P+\sigma_{xz}(P) - 2S$, which shows that on $E_1$ we have 
$2S = \sigma_{xz}(P)$, so by Proposition \ref{twotorsion}, the point $S$ 
has order $4$.  The other points can be handled similarly.
\end{proof}

The simplicity of the formulas in Proposition \ref{fourtorsion}
reflects the well-known fact that the Mordell-Weil groups of the
Jacobians of these fibrations over $\Q(i)(t)$ are both isomorphic to
$\Z/4\Z \oplus \Z/4\Z$.

\begin{prop}\label{atmostfour}
Assume that $P$ is defined over $\Q$ and that
all its coordinates are nonzero. 
If $(e_i(P)) - (P)$ is a torsion divisor class on $C_i$, its order is at most~$4$.
\end{prop}
\begin{proof}
Since the coordinates of $P$ are all nonzero, the $2$-torsion subgroup of the 
Mordell-Weil group of the Jacobian of $C_i$ over $\Q$ has order $4$ by 
Proposition \ref{twotorsion}. By Mazur's theorem,
the torsion subgroup of $C_i$ is therefore
$\Z/2\Z \oplus \Z/2k\Z$ where $k \le 4$.  Accordingly, if $(e_i(P)) - (P)$ is
torsion, then, since it is a multiple of $2$ by Proposition
\ref{phipistwo}, its order is at most~$4$.
\end{proof}
Of course the order of $(e_i(P))-(P)$ is $1$ if and only if $e_i(P) = P$;
that is, by Proposition \ref{zerofix}, if and only if one of the coordinates
of $P$ is $0$.  In this case, $P$ is fixed by both endomorphisms $e_k$,
and the construction of this paper does not allow us to
exhibit infinitely many rational points on $V$.

\begin{prop}\label{ordtwo}
Assume that $C_i$ is nonsingular. 
Then the exact order of the divisor $(e_i(P)) - (P)$
is $2$ if and only if $P$ lies on a line.
\end{prop}

\begin{proof}
Suppose that $P$ lies on a line $L$.  Then 
none of its coordinates is $0$, as the intersection of any line with 
any coordinate plane is contained in $\Omega$. By Proposition \ref{zerofix}
this implies $e_i(P) \neq P$, so $e_i(P)$ and $P$ are not linearly 
equivalent, as no two different points are linearly equivalent to each 
other on a curve of positive genus. The line $L$ is a component of $C_j$
with $j \neq i$, so by Proposition \ref{onlines}
we have $e_i(P) = \sigma_{uv}(P)$ for some coordinates $u,v$ of $\P^3$. 
By Proposition \ref{twotorsion} we find immediately that $(e_i(P))-(P)$
has order $2$, thus proving one implication.

Note that each $L$ of the $24$ lines in the fibres of $f_j$ intersect 
$C_i$ in two points by Corollary \ref{otherpoint}, which are different 
by Lemma \ref{nointerference}. This already gives $48$ points $S$ on 
$C_i$ with $2(e_i(S))\sim 2(S)$. By Proposition \ref{zerofix},
the $16$ intersection points of $C_i$ with the coordinate planes also 
satisfy $2(e_i(S))\sim 2(S)$, so that is $64$ points total. Note 
that $3(S)+e_i(S)$ is linearly equivalent to a hyperplane section $H$ 
on $C_i$ for every point $S$ on $C_i$ by Corollary \ref{hyp}, so we have 
$2(e_i(S))\sim 2(S)$ if and only if $8(S)$ is linearly equivalent to $2H$.
Since multiplication by $8$ on the Jacobian of $C_i$ has degree $64$, 
there are no points $S$ with $8(S) \sim 2H$ other than the ones already
described. This proves the converse. 
\end{proof}

\section{Proof of the main theorem}

\begin{define}
Let $k$ be a positive integer and $i \in \{1,2\}$.  Define $T_{ki}$ to
be the closure of the locus of points $P$ on $V$ for which $C_i$, the
fibre of $f_i$ through $P$, is nonsingular and where $(e_i(P)) - (P)$ 
is a divisor of exact order $k$ on $C_i$.
\end{define}

In these terms, Proposition \ref{zerofix} states that $T_{11} = T_{12}$
is the intersection of $V$ with the union of the coordinate planes.
Likewise, in Proposition \ref{ordtwo} we showed that $T_{21} \cup
T_{22}$ is contained in the union of the lines on $V$.
Therefore, if a point $P$ satisfies the hypotheses of our main
theorem (Theorem \ref{main}, repeated below as Theorem \ref{main2}), it does not lie in $T_{ki}$ for any
$k, i \in \{1,2\}$.

By Corollary \ref{orderthree}, we have that $T_{3i}$ is the closure of 
$e_i^{-1}(T_{1i})
\setminus T_{1i}$.  Now, $T_{1i}$ consists of the intersections of the
coordinate planes with $V$.  Calling these $X_1, X_2, X_3, X_4$, we
find that $e_i^{-1}(X_j) = X_j \cup Y_{ij}$, where the $Y_{ij}$ are
geometrically irreducible.  In other words, each $T_{3i}$ is the union
of four irreducible components.  Also, as in Proposition
\ref{fourtorsion}, the exact order of $e_i(P) - P$ is $4$ if and only
if $e_i(P) = \alpha(P)$, where $\alpha$ ranges over a certain set
of $12$ automorphisms.  So $T_{4i}$ is the union of $12$ such curves,
which are in fact geometrically irreducible.

\begin{theorem}\label{genus}
The set $T_{3i}$ is
the union of geometrically irreducible curves of geometric genus $41$, while
$T_{4i}$ is the union of geometrically irreducible curves of geometric
genus $13$.
\end{theorem}

\begin{proof} 
From the above, we can compute the $T_{3i}$ and $T_{4i}$ explicitly.
Using Magma, it is easy to verify that the geometric genus of each 
of the $12$ components of each $T_{4i}$ is $13$
and that the geometric genus of each of the $4$ components of each
$T_{3i}$ is $41$.
\end{proof}

\begin{prop}\label{notbothtor}
The intersection $(T_{31} \cup T_{41}) \cap (T_{32} \cup T_{42})$
does not contain any rational points outside the coordinate planes
and the $48$ lines on $V$.
\end{prop}

\begin{proof}
Suppose $P$ is a rational point in the given intersection
that is not on any of the coordinate planes, and set 
$Q = \iota_{a,b,c,d}(P) \in V_0$. Then the ratios of the fourth 
powers of the coordinates of $Q$ are rational. 
By computer calculation, we will see that 
up to the action of $G$, there is only one such point on $V_0$
with this property.
Working on $V_0$, we
intersect every component of $T_{31}$ and $T_{41}$ with every
component of $T_{32}$ and $T_{42}$, 
and remove extraneous points with one of the
coordinates $0$ or that lie on a line.
Then we resolve these schemes into primary components
over $\Q$. Since all of the intersections
have dimension $0$, we may use Magma to find all the $\Qbar$-points on these 
components. The only points for which the fourth powers of the coordinates
have rational ratios are in the orbit under $G$ (acting on $V_0$) of 
$(\eta : 1 : 1: 1)$ with $\eta^4 = -3$, so $Q$ is one of these points,
defined over $\Q(i,\eta)$, where $i$ denotes a square root of $-1$,
and $P$ is its inverse image under $\iota_{a,b,c,d}$. 
Let $\alpha,\beta,\gamma,\delta$ be the fourth roots of $a,b,c,d$ respectively 
that determine $\iota_{a,b,c,d}$, viewed as elements of a field
$K = \Q(i,\alpha,\beta, \gamma,\delta,\eta)$.
Fix an extension $v_3$ of the $3$-adic valuation of $\Q$ to $K$.
Note that $\eta$ is a uniformizer for
$v_3$; we normalize so that $v_3(\eta) = 1$.  Given a $K$-point $R
= (r_0:r_1:r_2:r_3)$ of $V$, let $\nu(R) = \sum_{i=0}^3 v_3(r_i)$
viewed as an element of $\Z/4\Z$.  It is clear that $\nu(R)$ is
well-defined, $G$-invariant, and $0$ if $R$ is defined over $\Q$.
However, $\nu(P) = v_3(\eta/(\alpha \beta \gamma \delta))$, and
from the fact that $abcd$ is a square it follows that
$v_3(\alpha \beta \gamma \delta)$ is even.  This is a contradiction,
because $v_3(\eta) = 1$ and $\nu(P)=0$.
We conclude that such a $P$ does not exist. 
\end{proof}

We are now ready to prove the main theorem, repeated here.

\main{main2}

\begin{proof} For $i=1,2$, let $C_i$ denote the fibre of $f_i$ through $P$,
endowed with the structure of an elliptic curve with $P$ as the origin.
By assumption, $e_i(P)$ does not have order $1$ or $2$ on either $C_i$.  
That being so, Proposition \ref{notbothtor} assures us
that for some $i$ the order of $e_i(P)$ is infinite.  Say (without
loss of generality) that this $i$ is 1. Then the subgroup $S$ of $C_1(\R)$ 
generated by $e_1(P)$ and the $2$-torsion points is infinite and, in fact,
dense in the real analytic topology. 
For each point $Q$ in $S$, consider the divisor class
$(e_2(Q))-(Q)$ on the fibre of $f_2$ passing through $Q$.  Its order is
$1$ or $2$ finitely often, by Propositions \ref{zerofix} and
\ref{ordtwo}; by Theorem \ref{genus} it is $3$ or $4$ finitely often,
because $C_1$ does not have genus $41$ or $13$ and so cannot be one of
the curves on which the order of $(e_2(R))-(R)$ is $3$ or $4$ for all $R$.
(It is not necessary to use Faltings' theorem here.)  
In other words, there are only finitely many points $Q$ in $S$ for which 
the fibre of $f_2$ through $Q$ contains only finitely many rational points. 
If $R\in S$ is not one of these finitely many points, then similarly the 
set of rational points on the fibre of $f_2$ through $R$ is infinite and 
dense in the real analytic topology.
Of course $C_1$ meets any fibre of $f_2$
in only finitely many points, so there are infinitely many distinct
fibres of $f_2$
with infinitely many rational points.
Zariski density follows. 

Now we treat the real analytic topology. 
If $a,b,c,d$ were all of the same sign, then $V=V_{a,b,c,d}$ would not
have any real points, so we conclude that not all of them have the same 
sign. Since $abcd$ is a nonzero square, it is positive, so two of $a,b,c,d$
are positive and two are negative. Without loss of generality, we assume 
that $a,d>0$ and $b,c<0$, and we choose real $\alpha,\beta,\gamma,\delta$ 
such that $\alpha^4=a$, $\beta^4=-b$, $\gamma^4=-c$, and $\delta^4=d$. 
Then over $\R$, one of the elliptic fibrations, say $f_1$, is given by 
$(x:y:z:w) \mapsto (\alpha^2x^2-\beta^2y^2 : \gamma^2 z^2-\delta^2 w^2)
= (\gamma^2 z^2+\delta^2 w^2:\alpha^2x^2+\beta^2y^2)$, 
up to a linear automorphism of $\P^1$. The fibration $f_2$ can be given 
by $(x:y:z:w) \mapsto (\alpha^2x^2-\beta^2y^2 : \gamma^2 z^2+\delta^2 w^2)
= (\gamma^2 z^2-\delta^2 w^2:\alpha^2x^2+\beta^2y^2)$. Let $L$ be the line 
defined by $\alpha x = \beta y$ and $\gamma z = \delta w$. Then $L$ is 
contained in the fibre of $f_2$ above $(0:1)$ and it is easy to check that
$L(\R)$ maps surjectively to $f_1(V(\R)) \subset \P^1(\R)$. 

We now show that there exists 
a nonempty open subset $A$ of $V(\R)$ in which the subset of rational points
is dense. The locus of points on $V$ where $f_1$ and $f_2$ do not give 
local parameters is of codimension $1$. Since the set of rational points 
is Zariski dense, we can choose a rational point $Q$ not on a line or 
a coordinate plane such that $f_1$ and $f_2$ give local parameters at $Q$. 
We choose $Q$ such that the set of rational points on the fibre $F$
of $f_1$ through $Q$ is dense in $F(\R)$ as well. 
Let $A\subset V(\R)$ be an open neighbourhood of $Q$ and $J_1,J_2 \subset 
\P^1(\R)$ connected open subsets such that the map $f=(f_1,f_2) \colon 
A \rightarrow J_1\times J_2$ is a homeomorphism. It suffices to show that 
$f(A\cap V(\Q))$ is dense in $J_1 \times J_2$. Set $s_i=f_i(Q)$, so that 
$f(Q) = (s_1,s_2)$.  Choose $(r_1,r_2) \in J_1\times J_2$. 
Since the rational points on $F$ are dense in $F(\R)$, following the 
proof of the density of rational points in the Zariski topology,
we can choose a rational $t_2\in J_2$, 
\footnote{it said $t_2 \in J_1$ before, but I think that was a typo,
also changed beginning of sentence}
arbitrarily close to $r_2$, such 
that $(s_1,t_2) = f(R)$ for some $R\in F(\Q)$ for which the rational 
points in the fibre $G$ of $f_2$ through $R$ are dense. Therefore, 
there is a $t_1$, arbitrarily close to $r_1$, such that $(t_1,t_2)=f(T)$ 
for some $T \in G(\Q)$. Since we can $(t_1,t_2)$ arbitrarily close to 
$(r_1,r_2)$, it follows that $V(\Q) \cap A$ is dense in $A$. 
The following diagram depicts the remainder of the argument.

\begin{figure}[htbp]
\begin{center}

\input{abcd.pstex_t}

\end{center}
\end{figure}

Let $I\subset \P^1(\R)$ be a nonempty connected open subset contained 
in $f_1(A)=J_1$. Suppose $B$ is a nonempty open subset of $V(\R)$ and let 
$J \subset \P^1(\R)$ be a connected open subset contained in $f_1(B)$. 
Since $f_1(L) = f_1(f_2^{-1}((0:1)))$ contains $I$ and $J$, for 
$t\in \P^1(\R)$ close enough to $(0:1)$ the set $f_1(f_2^{-1}(t))$ 
intersects both $I$ and $J$ nontrivially. Choose such a $t$ close to 
$(0:1)$, let $G_t$ denote the fibre $f_2^{-1}(t)$, and choose a nonempty 
connected open subset $I' \subset \P^1(\R)$ contained in 
$I \cap f_1(G_t(\R))$.
Since $V(\Q) \cap A$ is dense in $A$, we may choose $Q \in V(\Q) \cap A$
such that $a = f_1(Q)$ is contained in $I'$ and the set of rational points 
on the fibre $F_a = f_1^{-1}(a)$ is dense in $F_a(\R)$; moreover, such that 
the set $S$ of those rational points $R$ on $F_a$ for which the set of 
rational points on the fibre of $f_2$ through $R$ is dense, is itself dense 
in $F_a(\R)$. Since $a \in I'$ is contained in $f_1(G_t(\R))$, there is an
$X \in G_t(\R)$ with $f_1(X)=a$, so we can find $R \in S \subset F_a(\Q)$ 
such that $R$ is arbitrarily close to $X$ and thus $b=f_2(R)$ is arbitrarily 
close to $t$. Since $f_1(G_t(\R))$ intersects $J$ nontrivially, we may 
choose $R$ so close to $X$ that also $f_1(G_b(\R))$ intersects $J$ 
nontrivially,
with $G_b = f_2^{-1}(b)$. Since the set of rational points on $G_b$ is 
dense in $G_b(\R)$, we can find a point $T \in G_b(\Q)$ such that 
$c = f_1(T)$ is contained in $J$; 
moreover, we can pick $T$ so that the set of rational 
points on the fibre $F_c = f_1^{-1}(c)$ is dense in $F_c(\R)$.
Since $F_c(\R)$ intersects $B$ nontrivially and $F_c(\Q)$ is dense in 
$F_c(\R)$, we conclude that $B$ contains at least one rational point. 
Thus, any nonempty open subset of $V(\R)$ contains at least one rational 
point, and we conclude that $V(\Q)$ is dense in $V(\R)$. 
%
\end{proof}

The second theorem from the introduction, also repeated here, 
follows almost immediately.

\second{second2}

\begin{proof}
The surface $V_{a,b,c,d}$ contains the point $P=(1:1:1:1)$, which does 
not lie on a coordinate plane. Each of the $48$ lines on $V$ is contained 
in one of the sets $ax^4+by^4=0$, $ax^4+cz^4=0$, or $ax^4+dw^4=0$. Since 
no two of $a,b,c,d$ sum to zero, the point $P$ does not lie on any of the 
lines. By Theorem \ref{main}, the set of rational points 
of the surface is dense in both the Zariski and the real 
analytic topology.
\end{proof}

\begin{remark}
Theorem \ref{second} is included to give a large family of 
surfaces for which we can prove unconditionally that the set 
of rational points is dense. Each surface $V=V_{a,b,c,d}$ 
with $a+b+c+d=0$ contains the point $P=(1:1:1:1)$ and if $N^2=abcd$,
then $V$ also contains the less trivial point $Q=(x:y:z:w)$ with 
\begin{align}
x&=(3bc+ad)(a+d) + 4N(b-c),\nonumber\\
y&=(3ac+bd)(b+d) + 4N(c-a),\label{abcdeqsnorm}\\
z&=(3ab+cd)(c+d) + 4N(a-b),\nonumber\\
w&=-d(ab+ac+bc) - 9abc,\nonumber
\end{align}
which equals $e_i(P)$ for some $i\in \{1,2\}$ by (\ref{abcdeqs}).
Theorem \ref{second} appears weaker than Theorem \ref{main} 
because of the condition $a+b+c+d = 0$, but in fact Theorem \ref{main} 
follows directly from Theorem \ref{second}. Indeed, given a 
point $P'=(x_0:y_0:z_0:w_0)$ on $V'=V_{a',b',c',d'}$, the map 
$(x:y:z:w) \mapsto (x_0^{-1}x:y_0^{-1}y:z_0^{-1}z:w_0^{-1}w)$
sends $P'$ to $P=(1:1:1:1)$ and 
induces an isomorphism $\tau_{P'}$ from $V'$ to $V=V_{a,b,c,d}$
with $a=a'x_0^4$, $b=b'y_0^4$, $c=c'z_0^4$, and $d=d'w_0^4$,
satisfying $a+b+c+d=0$. The point $P'$ lies on a line in $V'$ 
if and only if $P$ lies on a line in $V$, which is the case if 
and only if two of $a,b,c,d$ sum to $0$. 
In a conversation, Andrew Granville reduced (\ref{abcdeqs}) to 
the equations in (\ref{abcdeqsnorm}) and noted that the 
endomorphism $e_i$ on $V'$ can be recovered from these 
simpler formulas, as we have $e_i(P') = \tau_{P'}^{-1}(Q)$.
\end{remark}

\begin{remark}
Without reference to the real analytic topology, Theorem \ref{main}
and its proof also apply to rational function fields over $\Q$. Take, 
for instance, the function field 
$K = \Q(a,b,c)$, set $d=-a-b-c$, and define $L=K[x]/(x^2-abcd)$. 
Then, as in Theorem \ref{second}, 
we find that $V_{a,b,c,d}(L)$ is Zariski dense in $V_{a,b,c,d}$.
\end{remark}

\section{General number fields}

Theorem \ref{main} does not generalize immediately to number fields,
as Mazur's theorem does not either. 
Samir Siksek pointed out to us that one can prove the following 
statement for general number fields. Note that Definition \ref{iota}
applies to any number field.

\begin{theorem}
There exists a Zariski open subset $U \subset V_{1,1,1,1}$, such that 
for each number field $K$ there exists an integer $n$, such that 
for all $a,b,c,d \in K^*$ with $abcd \in {K^*}^2$, if 
$\iota_{a,b,c,d}^{-1}(U) \subset V=V_{a,b,c,d}$ contains more than 
$n$ points over $K$, then the set of $K$-rational points on $V$ is 
Zariski dense. 
\footnote{proof adjusted to lose silly and unnecessary factor $2\cdot 16$ in $n$.
please check carefully. }
\end{theorem}
\begin{proof}
For each $P \in V$, let $o_i(P)\in \{1,2,3,\ldots\}\cup \{\infty\}$ 
denote the order of $e_i(P)$ 
on the fibre of $f_i$ through $P$ with $P$ as origin. We refer 
to $o_i(P)$ as the order of $e_i(P)$.

Recall that for any positive integer $N$, the curve $X_1(N)$ 
parametrizes pairs $(E,P)$, where $E$ is an elliptic curve
and $P$ is a point of order $N$. The genus of $X_1(N)$ is at 
least $2$ for $N=13$ and $N \geq 16$ (see \cite[p. 109]{oggrat}). 
For the remaining $N$, i.e., $N \in I:=\{1,\ldots,12,14,15\}$, and 
$i \in \{1,2\}$, 
let $T_{i,N}$ be the closure of the locus of all points $P$ on 
$V_{1,1,1,1}$ such that $o_i(P)=N$. Let $U \subset V_{1,1,1,1}$ be the 
complement of the $T_{i,N}$, so that for all $P \in U$ we have 
$o_i(P) \not \in I$. 

Suppose $K$ is a number field. By Merel's Theorem \cite[Corollaire]{merel},
there is an integer $B$, depending in fact only on the degree of $K$, 
such that any $K$-rational point of finite order on an elliptic curve 
over $K$ has order at most $B$. Set 
$$
s = \mathop{\sum_{N \leq B}}_{N \not \in I} \# X_1(N)(K).
$$
Note that $s$ is well defined because for each $N$ in the sum, the genus 
of $X_1(N)$ is at least $2$, so $\# X_1(N)(K)$ is finite by Faltings' 
Theorem. We conclude that up to isomorphism over the algebraic closure 
of $K$, there are at most $s$ 
elliptic curves over $K$ containing a point of finite order 
$N  \not \in I$. 

Take $a,b,c,d \in K^*$ with $abcd \in {K^*}^2$, and let $f_1,f_2$ be the 
elliptic fibrations of $V=V_{a,b,c,d}$ over $K$ as before. 
It is easy to check that the degree of the maps 
$j_i \colon \P^1 \rightarrow \P^1$ that send $t \in \P^1$ to the 
$j$-invariant of the fibre $f_i^{-1}(t)$ equals $24$.
Therefore there are at most $24s$ fibres of $f_i$, defined over $K$, of which 
the Jacobian contains a point over $K$ of finite order $N  \not \in I$. 
%
%
Take $n=24sB^2$ and assume $U_{a,b,c,d}=\iota_{a,b,c,d}^{-1}(U)$ 
contains more than $n$ points over $K$. Suppose that no fibre of $f_1$ 
contains more than $B^2$ points of $U_{a,b,c,d}(K)$. 
Then there would at least be one point 
$P\in U_{a,b,c,d}(K)$ on a fibre of $f_1$, say $C$, such that all 
$K$-rational torsion points on the Jacobian of $C$ have order in $I$.
From $P \in U_{a,b,c,d}$ we derive $o_i(P)\not \in I$, so $e_i(P)$
has infinite order and $C$ has infinitely many rational points. 
We conclude that there is a fibre of $f_1$, say $C_1$, with more 
than $B^2$ points of $U_{a,b,c,d}(K)$. By Merel's Theorem, at least 
one of these points has infinite order, so that there are infinitely 
many $K$-rational points on $C_1$. 

As $C_1$ intersects $U_{a,b,c,d}$ nontrivially, infinitely many of these 
rational points $Q$ lie in $U_{a,b,c,d}$, thus satisfying $o_2(Q) 
\not \in I$. Since at most $n$ points $Q$ on $V$ 
have finite order $o_2(Q) \not \in I$ on the fiber of $f_2$ through $Q$, 
we get $o_2(Q) = \infty$ for infinitely many rational $Q$ on $C_1$. 
It follows that infinitely many fibres of $f_2$ contain infinitely 
many rational points, so the set of rational points is Zariski dense. 
\end{proof}

\small
\nocite{*}
\bibliography{abcd}
\bibliographystyle{plain}

\end{document}